\documentclass[twoside,12pt]{article}
\usepackage{graphicx,amsmath,latexsym,amssymb,amsthm}
 \usepackage{epstopdf}
\font\chuto=cmbx10 at 16pt \font\kamy=lcmssb8
at 9pt

\newtheorem{theorem}{Theorem}[section]
\newtheorem{proposition}[theorem]{Proposition}
\newtheorem{definition}[theorem]{Definition}
\newtheorem{corollary}[theorem]{Corollary}

\newtheorem{example}[theorem]{Example}
\newtheorem{remark}[theorem]{Remark}

\numberwithin{equation}{section}

\begin{document}
\vskip1.5cm

\centerline {\bf \chuto New generalization of geodesic convex function }

\vskip.2cm

\centerline {\bf \chuto  }


\vskip.8cm \centerline {\kamy Ohud Bulayhan Almutairi $^a$\footnote{{\tt Corresponding author email:  Ohudbalmutairi@gmail.com ( Ohud Bulayhan Almutairi)}} and Wedad Saleh  $^{b}$} 

\vskip.5cm

\centerline {$^a$ Department of Mathematics, University of Hafr Al-Batin, Hafr Al-Batin 31991, Saudi Arabia}
\centerline {$^b$ Department of Mathematics, Taibah University, Al- Medina, Saudi Arabia  }



\vskip.5cm \hskip-.5cm{\small{\bf Abstract :}  
As a generalization of geodesic function, this paper introduces the notion of geodesic $ \varphi_{E} $-convex function. Some properties of $ \varphi_{E} $-convex function and geodesic $ \varphi_{E} $-convex function are established. The concepts of geodesic $ \varphi_{E} $-convex set and $ \varphi_{E} $-epigraph are equally given. The characterization of geodesic $ \varphi_{E} $-convex functions in terms of their $ \varphi_{E} $-epigraphs are also obtained.


\hrulefill


\section{Introduction}
\hskip0.6cm
Convex functions have been widely used in many branches of mathematics; for example, they can be seen in mathematical analysis, function theory, functional analysis and optimization theory \cite{Boltyanski,Danzer,AdemWedadEb2015,KolwankarGangal1999,ozdemir2013some}. Thus, the function can be defined as follows:\\

 A function $ h:U\subseteq \mathbb{R}\longrightarrow \mathbb{R} $ is  a convex if 
\begin{equation}\label{eq1}
h(\eta u_{1}+(1-\eta)u_{2})\leq \eta h(u_{1})+(1-\eta)h(u_{2}), \ \ \forall u_{1}, u_{2}\in U, \ \  \eta\in [0,1].
\end{equation}
If the inequality sign in (\ref{eq1}) is reversed then $ h $ is called  a concave on the set $ U $.


 For example in economics, for a production function $ u=h(L) $ , the concavity of $ h $ is expressed by saying that $ h $ exhibits diminishing returns. If $ h $ is convex, then  it exhibits increasing returns. On the other hand, many new problems in applied mathematics are encountered where the notion of  convexity is not enough to describe them in order to reach favourite results. For this reason, the concept of convexity has been extended and generalized in several studies, see \cite{OK,GrinalattLinnainmaa2011,Syau,kS,RuelAyres1999,Awan}. Due to the curvature and torsion of a Riemannian manifold, high non-linearity  appears in the study of convexity on such a manifold. Geodesics are length minimizing curves and the notion of a geodesic convex function arises naturally in a complete Riemannian manifold, which has been recently studied in \cite{Greene1981Convex,Udrist1994}.\\

  In 1999, an important generalization of the convex function, called the $ E $-convex function, was defined by Youness  \cite{Youness1999}. This type of function  has  some  applications in various branches of mathematical sciences \cite{Abou1999inequalities,Noor1994fuzzy}. On the other hand, Yang \cite{Yang2001} showed that  some results given by Youness  \cite{Youness1999} seem to be incorrect. Therefore, Chen \cite{Chen2002} extended E-convexity to a semi E-convexity and discussed some of its properties. For more results on E-convex function and semi $ E $-convex function one consult the following references  \cite{FulgaPreda,Iqbal,IAA2012,SyauLee2005}.  In 2012, Iqbal et al.\cite{IAA2012} introduced and studied a new class of convex sets together with functions that are called geodesic $ E $-convex sets and geodesic E-convex functions on Riemannian manifolds, these were extended to the  geodesic strongly $ E $-convex sets and geodesic strongly $ E $-convex functions in 2015 by Adem and Saleh \cite{AW}. Also, Iqbal et al.\cite{Iqbal} introduced 
   geodesic semi $ E $-convex funcions. Following these developments, Adem and Saleh \cite{AdemWedadEb2015} introduced
    geodesic semi $ E $-$ b $-vex (GSEB) functions through which some properties were discussed.\\
  
 Other developments include the work of Eshagh et. al. \cite{Gordji2016}, who introduced the notion of a $ \varphi $-convex function in 2016. They equally studied Jensen and Hermite-Hadamard type inequalities related to this function. Moreover, the notion of  $ \varphi_{E} $-convex functions was defined as the generalization of $ \varphi $-convex functions. Absos et al. further introduced the notion of a geodesic $ \varphi $- convex function through which some basic properties of this function were studied \cite{Shaikh2018phicomvex}.\\

   The paper is organized as follows. Section 2 deals with the rudimentary facts of convex functions and convex sets. Section 3 is devoted to the study of some properties of $ \varphi_{E} $-convex functions. In Section 4, we discuss a new class of functions on Riemannian manifolds, which is called the geodesic $ \varphi_{E} $- convex function. Some of the properties of this function are also studied. In Section 5, the characterization of geodesic $ \varphi_{E} $-convex functions through their corresponding $ \varphi_{E} $-epigraphs is reported.
\section{Preliminaries}
This section provide some definitions and properties that can be later used in the study in order to report our results. Thus, several definitions and properties of the real number set and the Riemannian manifold can be found in many differential geometry books and papers \cite{Udrist1994}. 
Throughout this paper, we consider an interval $ U=[u_{1},u_{2}] $ in $ \mathbb{R} $ and $  \varphi\colon \mathbb{R} \times \mathbb{R} \longrightarrow \mathbb{R} $  is a bifunction.
\begin{definition}\cite{Gordji2016}
	A function $ h: U\longrightarrow\mathbb{R} $ is called $ \varphi $-convex if 
	\begin{equation}\label{eq2}
	h(tu_{1}+(1-t)u_{1})\leq h(u_{2})+t\varphi(h(u_{1}), h(u_{2}),
	\end{equation}
	$ \forall u_{1},u_{2}\in U , t\in [0,1] $
\end{definition}
 In the above definition, if $ \varphi(h(u_{1}), h(u_{2}))=h(u_{1})-h(u_{2}) $, then inequality(\ref{eq2}) becomes  inequality (\ref{eq1}).
\begin{definition}\cite{Gordji2016}
	The function $  \varphi\colon \mathbb{R} \times \mathbb{R} \longrightarrow \mathbb{R} $ is said to be 
\begin{enumerate}
	\item   nonnegatively homogeneous if $$ \varphi(tu_{1},tu_{2})=t\varphi(u_{1},u_{2}),\forall u_{1},u_{2}\in \mathbb{R}, t\geqslant 0. $$
	\item additive if $$ \varphi(u_{1}+v_{1},u_{2}+v_{2})=\varphi(u_{1},u_{2})+\varphi(v_{1},v_{2}),\forall u_{1},u_{2},v_{1},v_{2}\in \mathbb{R}. $$
	\item nonnegatively linear if  $ \varphi $ is both nonnegatively homogeneous and additive.
\end{enumerate}
\end{definition}

 \begin{definition}\cite{Youness1999}
 A set $ U\subset \mathbb{R}^{n} $ is called E-convex set if there is a mapping $ E:\mathbb{R}^{n}\longrightarrow \mathbb{R}^{n} $ such that $$ t(E(u_{1}))+(1-t)E(u_{2})\in U,\ \ \forall u_{1},u_{2}\in U, t\in [0,1] $$
 \end{definition}
\begin{definition}\cite{Gordji2016,Youness1999}
Assume that $ U\subset \mathbb{R}^{n} $ is an E-convex set, then the function $ h:U\longrightarrow \mathbb{R} $ is called
\begin{enumerate}
	\item an E-convex function, if 
		\begin{equation}
	h(tE(u_{1})+(1-t)E(u_{1}))\leq th(E(u_{1}))+(1-t) h(E(u_{2})),
	\end{equation}
	$ \forall u_{1},u_{2}\in U , t\in [0,1] $
	\item $ \varphi_{E} $-convex function, if 
		\begin{equation}\label{eq3}
	h(tE(u_{1})+(1-t)E(u_{1}))\leq h(E(u_{2}))+t\varphi(h(E(u_{1})), h(E(u_{2})),
	\end{equation}
	$ \forall u_{1},u_{2}\in U , t\in [0,1] $
\end{enumerate}
\end{definition}
If $ \varphi(h(E(u_{1})), h(E(u_{2}))= h(E(u_{1}))- h(E(u_{2}))$ in inequality(\ref{eq3}), then we obtain the E-convex function.\\

Now, let $ (N,g) $ be a complete $ m $-dimensional Riemannian manifold with Riemannian connection $ \bigtriangledown $. Given a piecewise $ C^{1} $ path $ \gamma:[\mu_{1},\mu_{2}]\longrightarrow N $ joining $ a_{1} $ to $ a_{2} $, that is, $ \gamma(\mu_{1})=a_{2} $ and $ \gamma(\mu_{2})=a_{1} $, the length of $ \gamma $ is defined by $$L(\gamma)=\int_{u_{1}}^{u_{2}}\|\acute{\gamma}(\lambda)\|_{\gamma(\lambda)}d\lambda.$$
For any two points $ a_{1},a_{2}\in N $, we define
$$d(a_{1},a_{2})=\inf\left\lbrace L(\gamma): \gamma \,\,\, \rm {is\, a\, piecewise\, C^{1} \,path\, joining}\,\,\, a_{1} \ {\rm to} \ a_{2} \right\rbrace. $$ Then $ d $ is a metric which induces the original topology on $ N $. \\

For every Riemannian manifold there is a unique determined Riemannian connection, called a Levi-Civita connection, denoted by $ \bigtriangledown_{X}Y $, for any vector fields $ X,Y \in N $. Also, a smooth
path $ \gamma  $is a geodesic if and only if its tangent vector is a parallel vector field along the path
$ \gamma $, i.e., $ \gamma  $ satisfies the equation $ \bigtriangledown_{\gamma'}\gamma'=0 $. Any path $ \gamma $ joining $ \mu_{1} $ and $ \mu_{2} $ in $ N $ such that
$ L(\gamma)=d(\mu_{1},\mu_{2}) $ is a geodesic and is called a minimal geodesic. Finally, let $ N $ as  a $ C^{\infty} $ complete $ n $-dimensional Riemannian manifold with metric $ g $ and Levi-Civita connection $ \bigtriangledown $. Moreover, considering that the points $ \mu_{1},\mu_{2}\in N $ and $ \gamma\colon[0,1]\longrightarrow N $ is a geodesic joining $ \mu_{1},\mu_{2} $, i.e., $ \gamma_{\mu_{1},\mu_{2}}(0)=\mu_{2} $ and $ \gamma_{\mu_{1},\mu_{2}}(1)=\mu_{1} $.
\begin{definition}\cite{Nicolaescu2020}
 Assume that $ N_{1}, N_{2} $ are smooth manifolds. A map $ h:N_{1}\longrightarrow N_{2} $ is a diffeomorphism if it is smooth, bijective, and the inverse $ h^{-1} $ is smooth.
\end{definition}
\begin{definition}\cite{Udrist1994}.
	A subset $ U\subseteq N $ is called t-convex if and only if $ U $ contains every geodesic $ \gamma_{\mu_{1},\mu_{2}} $ of $ N $ whose endpoints $ \mu_{1} $ and $ \mu_{2} $ are in $ B $.
\end{definition}
\begin{remark}
	If $ U_{1} $ and $ U_{2} $ are t-convex sets, then $ U_{1}\cap U_{2}$ is t-convex set , but $ U_{1}\cup U_{2} $ is not necessarily t-convex set.
\end{remark}
\begin{definition}\cite{Udrist1994}.
	A function $ h: U\subset N\longrightarrow \mathbb{R} $ is called  geodesic convex if and only if for all geodesic arcs $ \gamma_{\mu_{1},\mu_{2}} $, then $$h(\gamma_{\mu_{1},\mu_{2}}(t))\leq t h(\mu_{1})+(1-t)h(\mu_{2})$$ for each $ \mu_{1},\mu_{2}\in U $ and $t\in [0,1] $.
\end{definition}


\begin{definition}\cite{IAA2012}
	A set $ U\subset N $ is  geodesic E-convex where $ E: N\longrightarrow N $, if and only if there exists a unique geodesic $ \gamma_{E(\mu_{1}),E(\mu_{2})}(t)  $of length $ d(\mu_{1},\mu_{2}) $ which belongs to $ U $ for every $ \mu_{1},\mu_{2}\in U $ and $ t\in [0,1] $.
\end{definition}
\begin{definition}\cite{IAA2012,Shaikh2018phicomvex}
	A function $ h: U\longrightarrow \mathbb{R} $ is called	
	\begin{enumerate}
		\item 
 geodesic $ E $-convex if $ U $ is geodesic $ E $-convex set and 
	$$h(\gamma_{E(\mu_{1}),E(\mu_{2})})\leq th(E(\mu_{1}))+(1-t)h(E(\mu_{2})), \forall \mu_{1},\ \ \mu_{2}\in U ,\ \ t\in[0,1].$$
	\item geodesic $ \varphi $ -convex if $ U $ is a totally convex set and 
	 $$h(\gamma_{\mu_{1},\mu_{2}})\leq h(\mu_{2})+t\varphi\left(h(\mu_{1}), h(\mu_{2})\right) , \forall \mu_{1},\ \ \mu_{2}\in U ,\ \ t\in[0,1].$$
		\end{enumerate}
\end{definition}
\section{Some Properties of $ \varphi_{E} $-convex Functions}
This part of the work deals with some properties of $ \varphi_{E} $-convex functions. Considering that $ h:B\longrightarrow \mathbb{R} $ is $ \varphi_{E} $-convex function and $ E:\mathbb{R}\longrightarrow \mathbb{R} $, we present the following. For any two points $ E(\mu_{1}), E(\mu_{2})\in B $ with $ E(\mu_{1})< E(\mu_{2}) $ and for each point $ E(\mu)\in (E(\mu_{1}), E(\mu_{2})) $ can be expressed as 
$$E(\mu)=tE(\mu_{1})+(1-t) E(\mu_{1}), \ \ t=\frac{E(\mu_{2})-E(\mu)}{E(\mu_{2})-E(\mu_{1})}.$$
Also, since a function $ h $ is $ \varphi_{E} $-convex function if 
$$h(E(\mu))\leq h(E(\mu_{2}))+\frac{E(\mu_{2})-E(\mu)}{E(\mu_{2})-E(\mu_{1})} \varphi(h(E(\mu_{1})),h(E(\mu_{2}))),$$
then 
\begin{eqnarray}\label{eq4}
\frac{h(E(\mu_{2}))-h(E(\mu))}{E(\mu_{2})-E(\mu)} \geqslant \frac{\varphi(h(E(\mu_{1})),h(E(\mu_{2})))}{E(\mu_{1})-E(\mu_{2})}
\end{eqnarray}
$ \forall E(\mu)\in \left(E(\mu_{2}),E(\mu_{1}) \right)  $.\\
Hence, we can say that a function $ h $ is $ \varphi_{E}$-convex function
if it satisfies the inquality (\ref{eq4}).\\

The next example shows that $ \varphi_{E}$-convex function is not necessarily to be $ \varphi $- convex function.
\begin{example}
	Consider \begin{eqnarray*}
		h(u_{1}) &=& \begin{cases}
			1 ; u_{1}\geqslant 0,\\ 
			-u_{1}^{2} ; u_{1} <0,
		\end{cases}
	\end{eqnarray*}
with $ E(u_{1})=-a $ where $ a\in \mathbb{R}^{+} $ and $ \varphi(u_{1},u_{2})=u_{1}-2u_{2} $. Then 
$ h(tE(u_{1})+(1-t)E(u_{2}))= -a^{2} $
while $ h(E(u_{2}))+t\varphi(h(E(u_{1})),h(E(u_{2})))=(t-1)a^{2} $, which means that $ h $ is a $ \varphi_{E} $-conex function. On the other hand, if we take $ u_{1}>0 $ and $ u_{2}>0 $, then $ h $ is not $ \varphi $-convex function.
\end{example}
\begin{theorem}
	 If $ h:B\subset E(\mathbb{R})\longrightarrow \mathbb{R} $ is  differentiable and $ \varphi_{E}$-convex function in $ B $ and $ h(E(u_{1}))\neq h(E(u_{2})) $, then there are $ E(\alpha), E(\beta)\in (E(u_{2}),E(u_{1}))\subset B $ such that 
	 $$h'\left(E(\alpha) \right)\geqslant\frac{\varphi(h(E(u_{1})),h(E(u_{2})))}{h(E(u_{1}))-h(E(u_{2}))}h'(E(\beta))\geqslant h'(E(\beta)) .$$
\end{theorem}
\begin{proof}
	Since h is $ \varphi_{E}$-convex function, then 
	\begin{eqnarray}\label{eq5}
	\frac{h(E(u_{2}))-h(E(u))}{E(u_{2})-E(u)} &\geqslant& \frac{\varphi(h(E(u_{1})),h(E(u_{2})))}{E(u_{1})-E(u_{2})}\nonumber\\&& \hspace{-1in}=  \frac{\varphi(h(E(u_{1})),h(E(u_{2})))}{h(E(u_{1}))-h(E(u_{2}))}\times \frac{h(E(u_{1}))-h(E(u_{2}))}{E(u_{1})-E(u_{2})}.
	\end{eqnarray}
	Now applying the mean value theorem, then the inequality (\ref{eq5}) can be written as
	\begin{eqnarray}\label{eq6}
	h'\left(E(\alpha) \right)\geqslant\frac{\varphi(h(E(u_{1})),h(E(u_{2})))}{h(E(u_{1}))-h(E(u_{2}))}h'(E(\beta)),
	\end{eqnarray}
	for some $ E(\alpha)\in (E(u_{1}), E(u))\subset (E(u_{1}), E(u_{2})) $ and $ E(\beta)\in (E(u_{1}), E(u_{2})) $. Since $ \varphi(h(E(u_{1})),h(E(u_{2})))\geqslant h(E(u_{1}))-h(E(u_{2})) $, then the inequality (\ref{eq6}) yields
	\begin{eqnarray}
	h'\left(E(\alpha) \right)\geqslant\frac{\varphi(h(E(u_{1})),h(E(u_{2})))}{h(E(u_{1}))-h(E(u_{2}))}h'(E(\beta))\geqslant h'(E(\beta)).\nonumber
	\end{eqnarray}
\end{proof}
\begin{theorem}
	Assume that $ h:B\longrightarrow \mathbb{R} $ is a differentiable $ \varphi_{E}$-convex function. Then  $\forall E(\mu_{i})\in B,\ \ i=1,2,3 $ such that $ E(\mu_{1})<E(\mu_{2})<E(\mu_{3}) $, the following inequality holds 
	$$h'(E(\mu_{2}))+h'(E(\mu_{3}))\leq \frac{\varphi(h(E(\mu_{1})),h(E(\mu_{2})))+\varphi(h(E(\mu_{2})),h(E(\mu_{3})))}{E(\mu_{1}), E(\mu_{3})}.$$
\end{theorem}
\begin{proof}
	Since $ h $ is $ \varphi_{E}$-convex in each interval $ W_{1}=[E(\mu_{1}), E(\mu_{2})] $ and $ W_{2}=[E(\mu_{2}), E(\mu_{3})] $, hence
	\begin{eqnarray}\label{eq7}
	h(tE(\mu_{1})+(1-t) E(\mu_{2}))\leq h(E(\mu_{2}))+t\varphi(h(E(\mu_{1})),h(E(\mu_{2})))
	\end{eqnarray}
	and 
	\begin{eqnarray}\label{eq8}
	h(tE(\mu_{2})+(1-t) E(\mu_{3}))\leq h(E(\mu_{3}))+t\varphi(h(E(\mu_{2})),h(E(\mu_{3}))).
	\end{eqnarray}
	From inequalites(\ref{eq7}) and (\ref{eq8}), we get 
	\begin{eqnarray}&&\hspace{-1.2in}
	\frac{h(tE(\mu_{1})+(1-t) E(\mu_{2}))-h(E(\mu_{2}))+h(tE(\mu_{2})+(1-t) E(\mu_{3}))-h(E(\mu_{3}))}{t}\nonumber\\ \hspace{1in}&&\leq \varphi(h(E(\mu_{1})),h(E(\mu_{2})))+ \varphi(h(E(\mu_{2})),h(E(\mu_{3}))).\nonumber
	\end{eqnarray}
	Now, setting $ t\longrightarrow 0 $, we get
	\begin{eqnarray}\hspace{-1in}&&\label{eq9}
	h'(E(\mu_{2}))(E(\mu_{1})-E(\mu_{2}))+	h'(E(\mu_{3}))(E(\mu_{2})-E(\mu_{3})) \nonumber\\&&\leq \varphi(h(E(\mu_{1})),h(E(\mu_{2})))+ \varphi(h(E(\mu_{2})),h(E(\mu_{3}))).
	\end{eqnarray}
	Also, $ E(\mu_{3})> E(\mu_{2}) $ and $ E(\mu_{2})> E(\mu_{1}) $, which means that 
	$E(\mu_{1})-E(\mu_{3})<E(\mu_{1}) -E(\mu_{2}) $ and 	$E(\mu_{1})-E(\mu_{3})<E(\mu_{2}) -E(\mu_{3}) $, then 
	\begin{eqnarray}\hspace{.08in}&&\label{eq10}
	(E(\mu_{1})-E(\mu_{3}))(h'(E(\mu_{2}))+h'(E(\mu_{3})))\nonumber\\\hspace{-2in}&&\leq (E(\mu_{1})-E(\mu_{2}))h'(E(\mu_{2}))+(E(\mu_{2})-E(\mu_{3}))h'(E(\mu_{3})).
	\end{eqnarray}
Hence, from inequalities (\ref{eq9}) and (\ref{eq10}), we get the required result.	
\end{proof}
\section{Properties of Geodesic $ \varphi_{E} $-convex Functions}
In this section, we assume that  $ N $ as  a $ C^{\infty} $ complete $ n $-dimensional Riemannian manifold with Riemannian connection $ \bigtriangledown $. Let $ \mu_{1},\mu_{2}\in N $ and $ \gamma\colon[0,1]\longrightarrow N $ is a geodesic joining $ \mu_{1},\mu_{2} $, i.e., $ \gamma_{\mu_{1},\mu_{2}}(0)=\mu_{2} $ and $ \gamma_{\mu_{1},\mu_{2}}(1)=\mu_{1} $, and $ E $ is a mapping such that $ E:N\longrightarrow N $. Also, in this section, we give a definition of geodesic $ \varphi_{E} $- convex function in a Riemannian manifold $ N $ and study some of its properties.
\begin{definition}
		A function $ h: B\longrightarrow \mathbb{R} $ is 	
		 geodesic $ \varphi_{E} $ -convex if $ B $ is also geodesic E-convex set and 
			$$h(\gamma_{E(\mu_{1}),E(\mu_{2})})\leq h(E(\mu_{2}))+t\varphi\left(h(E(\mu_{1})), h(E(\mu_{2}))\right) , \forall \mu_{1},\ \ \mu_{2}\in B ,\ \ t\in[0,1].$$
			If the above inequality is strictly hold for all $ \mu_{1}, \mu_{2}\in B,  E(\mu_{1}) \neq  E(\mu_{2}), t\in[0,1] $, then $ h $ is called a strictly  geodesic $ \varphi_{E} $ -convex function.
	\end{definition}
\begin{remark}
If $ E $ is indenty map in the above defintion, then we have geodesic $ \varphi $-convex function. Morover, if $$ \varphi\left(h(E(\mu_{1})), h(E(\mu_{2}))\right)=\left(h(E(\mu_{1}))-h(E(\mu_{2}))\right), $$ then we have geodesic E-convex function. 
\end{remark}
\begin{theorem}
	Considering that $ B\subset N $ is an $ E $-convex set, we then write a function $ h:B\longrightarrow \mathbb{R} $ which is a geodesic $ \varphi_{E} $-convex if and only if the function $ K:h(\gamma_{E(u_{1}),E(u_{2})})$ is $ \varphi_{E} $-convex on $ I=[0,1] $.
\end{theorem}
\begin{proof}
	Let $ K $ be $ \varphi_{E} $-convex on $ I $, then 
		\begin{eqnarray}
	K(tE(\mu_{1})+(1-t) E(\mu_{2}))\leq K(E(\mu_{2}))+t\varphi(K(E(\mu_{1})),K(E(\mu_{2})))
	\end{eqnarray}
holds.\\
Also, let $ E(\mu_{1})=1, E(\mu_{2})=0, $ then 
$ K(t)\leq K(0)+t\varphi(K(1),K(0)). $
Hence 
$$h(\gamma_{E(\mu_{1}),E(\mu_{2})})\leq K(E(\mu_{2}))+t\varphi\left(K(E(\mu_{1})), K(E(\mu_{2}))\right).$$
Conversely, assume that $ h $ is geodesic $ \varphi_{E} $-convex function. By restricting the domain of $ \gamma_{E(\mu_{1}),E(\mu_{2})} $ to $ [\eta_{1},\eta_{2}] $, and hence the paramatrize form of this restriction can be rewritten as 
$$ \alpha(t)=\gamma_{E(\mu_{1}),E(\mu_{2})}(tE(\mu_{1})+(1-t)E(\mu_{2})) $$	
$$\alpha(0)=\gamma_{E(\mu_{1}),E(\mu_{2})}(E(\mu_{2})).$$
\end{proof}
\begin{proposition}
	\begin{enumerate}
		\item If $ h:B\longrightarrow \mathbb{R} $ is geodesic $ \varphi_{E} $-convex function where $ \varphi $ is non-negative linear, then $ xh: B\longrightarrow \mathbb{R},\ \ \forall x\geqslant 0 $ is also geodesic $ \varphi_{E} $-convex.
		\item Let $ h_{i}:  B\longrightarrow \mathbb{R},\ \ i=1,2  $ be two geodesic $ \varphi_{E} $-convex functions where $ \varphi $ is additive, then $ h_{1}+h_{2} $ is also geodesic $ \varphi_{E} $-convex function.
	\end{enumerate}
\end{proposition}
\begin{theorem}
	Suppose that $ B\subset N $ is a  geodesic $ E $-convex  set, $ h_{1}:B\longrightarrow\mathbb{R} $ is a geodesic $ E $-convex function and $ h_{2}: U\longrightarrow \mathbb{R} $ is a non-decreasing $ \varphi_{E} $-convex function such that $ Rang(h_{1})\subseteq U $. Then $ h_{2}oh_{1} $ is also a geodesic $ \varphi_{E} $-convex.
\end{theorem}
\begin{proof} The above theorem can be proved in the following way
	\begin{eqnarray}
	h_{2}oh_{1}(\gamma_{E(\mu_{1}),E(\mu_{2})})&=&	h_{2}\left( h_{1}(\gamma_{E(\mu_{1}),E(\mu_{2})})\right) \nonumber\\ &\leq& h_{2}\left(h_{1} (E(\mu_{2}))+t\varphi\left(h_{1}(E(\mu_{1})), h_{1}(E(\mu_{2}))\right)\right) \nonumber\\ &\leq& h_{2}\left(h_{1} (E(\mu_{2}))\right) +t\varphi\left(h_{2}\left( h_{1}(E(\mu_{1}))\right) , h_{2}\left(h_{1}(E(\mu_{2}))\right)\right)\nonumber\\ &=& h_{2}oh_{1}(E(\mu_{2}))+t\varphi\left(  h_{2}o h_{1}(E(\mu_{1})), h_{2}o h_{1}(E(\mu_{2})) \right). \nonumber   
	\end{eqnarray}
	Thus, $ h_{2}oh_{1} $ is geodesic $\varphi_{E} $-convex function.
\end{proof}
\begin{theorem}
Suppose that $ h_{i}:B\subset N\longrightarrow \mathbb{R}, \ \ i=1,2,\cdots, n$ are geodesic $ \varphi_{E} $-convex functions and $ \varphi $ is non-negatively linear. Then the function $ h=\sum_{i=1}^{n} x_{i}h_{i} $ is also a geodesic $ \varphi_{E} $-convex, $ \forall x_{i}\in \mathbb{R} $ and $ x_{i}\geqslant 0 $.
\end{theorem}
\begin{proof}
	Considering $ \mu_{1},\mu_{2}\in B $ and since $ h_{i}, i=1,2,\cdots ,n $ are geodesic $ \varphi_{E} $-convex functions, then
	\begin{eqnarray}
	h_{i}(\gamma_{E(\mu_{1}),E(\mu_{2})})\leq h_{i} (E(\mu_{2}))+t\varphi\left(h_{i}(E(\mu_{1})), h_{i}(E(\mu_{2}))\right)\nonumber.
	\end{eqnarray} 
	Also, 
	\begin{eqnarray}
x_{i}h_{i}(\gamma_{E(\mu_{1}),E(\mu_{2})})\leq x_{i}h_{i} (E(\mu_{2}))+t\varphi\left(x_{i}h_{i}(E(\mu_{1})), x_{i}h_{i}(E(\mu_{2}))\right)\nonumber.
\end{eqnarray} 	
Hence,
 \begin{eqnarray}
 \sum_{i=1}^{n}x_{i}h_{i}(\gamma_{E(\mu_{1}),E(\mu_{2})})\leq \sum_{i=1}^{n}\left[  x_{i}h_{i} (E(\mu_{2}))+t\varphi\left(x_{i}h_{i}(E(\mu_{1})), x_{i}h_{i}(E(\mu_{2}))\right)\right] \nonumber,
 \end{eqnarray}
which means that 
\begin{eqnarray}
h(\gamma_{E(\mu_{1}),E(\mu_{2})})\leq h (E(\mu_{2}))+t\varphi\left(h(E(\mu_{1})), h(E(\mu_{2}))\right) \nonumber.
\end{eqnarray}
\end{proof}
 Now we consider that $ N_{1} $ and $ N_{2} $ are two complete Riemannian manifolds, and $ \bigtriangledown $ is the Levi-Civita connection on $ N_{1} $. If $ H: N_{1}\longrightarrow N_{2} $ is a diffeomorphism, then $ H\ast\bigtriangledown =\bigtriangledown^{\ast} $ is an affine connection of $ N_{2} $. Moreover, let  $ \gamma  $ be a geodesic in $ (N_{1},\bigtriangledown) $ , then $ Ho\gamma $ is also a geodesic in $ (N_{2},\bigtriangledown^{\ast}) $ see \cite{Udrist1994}.
 \begin{theorem}\label{th1}
 Suppose that $ h:B\longrightarrow\mathbb{R} $ is a geodesic $ \varphi_{E} $-convex function and  $ H: N_{1}\longrightarrow N_{2} $, then a sufficent condition for $ ho H^{-1}: H(B)\longrightarrow \mathbb{R} $ to be a geodesic $ \varphi_{E} $-convex function is $ H $ must be a differmorphism. 
 \end{theorem}
\begin{proof}
	Assume that $ \gamma_{E(\mu_{1}),E(\mu_{2})} $ is a geodesic joining $ E(\mu_{1}) $ and $ E(\mu_{2}) $, where $ \mu_{1}, \mu_{2}\in B $. Since $ H $ is a diffeomorphism, then $ H(B) $ is totally geodesic, and $ Ho\gamma_{E(\mu_{1}),E(\mu_{2})} $ is geodesic joining $ H(E(\mu_{1})) $ and $ H(E(\mu_{2})) $. Then 
	\begin{eqnarray}&&\hspace{-0.5in}
	(hoH^{-1})\left(Ho\gamma_{E(\mu_{1}),E(\mu_{2})} (t)\right)\nonumber\\ \hspace{0.01in}&&= h(\gamma_{E(\mu_{1}),E(\mu_{2})} (t))\nonumber\\ \hspace{0.01in}&&\leq h (E(\mu_{2}))+t\varphi\left(h(E(\mu_{1})), h(E(\mu_{2}))\right)\nonumber\\ \hspace{0.01in}&& = (ho H^{-})(H (E(\mu_{2})))+t\varphi\left((ho H^{-})(E(\mu_{1})), (ho H^{-})(E(\mu_{2}))\right)\nonumber.   
	\end{eqnarray}	
\end{proof}
\begin{theorem}\label{th2}
	Assume that $ h:B\longrightarrow\mathbb{R} $ is a geodesic $ \varphi_{E} $-convex function, and $ \varphi $ bounded from above on $ h(B)\times h(B) $	 with an upper bound $ K $. Then $ h $ is continuous on $ Int(B) $.
\end{theorem}
\begin{proof}
Assume that $ E(\mu^{*})\in Int(B) $, then there exists an open ball $ B(E(\mu^{*}), r)\subset Int(B) $ for some $ r>0 $. Let us choose $ s $ where $ (0<s<r) $ such that the closed ball $ \bar{B}(E(\mu^{*}), s+\varepsilon)\subset B(E(\mu^{*}), r) $ for some arbitrary small $ \varepsilon>0 $. Choose any $ E(\mu_{1}),E(\mu_{2})\in\bar{B}(E(\mu^{*}), s) $. Put $ E(\mu_{3})=E(\mu_{2})+\frac{\varepsilon}{\|\mu_{2}-\mu_{1}\|}(E(\mu_{2})-E(\mu_{1})) $ and $ t=\frac{\|\mu_{2}-\mu_{1}\|}{\varepsilon+\|\mu_{2}-\mu_{1}\|} $. Then it is obvious that $ E(\mu_{3})\in \bar{B}(E(\mu^{*}), s+\varepsilon) $ and $ E(\mu_{2})=tE(\mu_{3})+(1-t)E(\mu_{1}) $. Thus,
\begin{eqnarray}
h(E(\mu_{2}))\leq h(E(\mu_{1}))+t\varphi (h(E(\mu_{3})),h(E(\mu_{1}))) \leq h(E(\mu_{1}))+tk.\nonumber
\end{eqnarray} 
Then, the above inequality can be written as 
$$ h(E(\mu_{2}))-h(E(\mu_{1}))\leq tk\leq \frac{\|\mu_{2}-\mu_{1}\|}{\varepsilon}k= L\|E(\mu_{2})-E(\mu_{3})\|, $$ where $ L=\frac{k}{\varepsilon} $. \\ Moreover, $$ h(E(\mu_{1}))-h(E(\mu_{2}))\leq L\|E(\mu_{2})-E(\mu_{3})\|, $$ Then $$ \left| h(E(\mu_{1}))-h(E(\mu_{2}))\right\| \leq L\|E(\mu_{2})-E(\mu_{3})\|, $$  and since $ \bar{B}(E(\mu^{*}), s) $ is arbitrary, then $ h $ is continuous on $ Int(B) $.  
\end{proof}
\begin{definition}\label{de1}
	A bifunction $ \varphi: \mathbb{R}^{2}\longrightarrow \mathbb{R} $ is called sequentialy upper bounded with respect to $ E $ if $$ \sup_{i}\varphi \left( E(u_{i})-E(v_{i})\right) \leq \varphi \left( \sup_{i}E(u_{i}),\sup_{i}E(v_{i})\right) $$ for any two bounded real sequences $ \left\lbrace E(u_{i})\right\rbrace, \left\lbrace E(v_{i})\right\rbrace  $.
\end{definition}
\begin{remark}
If $ E $ is  an indenty mapping in Definition \ref{de1}, then a bifunction $ \varphi:\mathbb{R}^{2}\longrightarrow \mathbb{R} $  is called sequentialy upper bounded \cite{Shaikh2018phicomvex}.
\end{remark}
\begin{proposition}
	Suppose that $ B\subset N $ is a geodesic $ E $-convex set, and $ \{h_{i}\}_{i\in \mathbb{N}} $ are non-empty family of geodesic $ \varphi_{E} $-convex function on $ B $, where $ \varphi_{E} $ is sequentially upper bounded w.r.t.$ E $. If $ \sup_{i}h_{i}(u) $ exist for each $ u\in B $, then $ h(u)=\sup_{i}h_{i}(u) $ are also geodesic $ \varphi_{E} $-convex functions.
\end{proposition}
\begin{proof}
	Let $  E(u_{1}), E(u_{2})\in B $, then 
	\begin{eqnarray}
	h\left(\gamma_{E(u_{1}),E(u_{2})} (t)\right)&=&  \sup_{i} h_{i}(\gamma_{E(u_{1}),E(u_{2})} (t))\nonumber\\ &\leq& \sup_{i} h_{i} (E(u_{2}))+t\sup_{i}\varphi\left(h_{i}(E(u_{1})), h_{i}(E(u_{2}))\right)\nonumber\\ &\leq& \sup_{i} h_{i} (E(u_{2}))+t\varphi\left(\sup_{i}h_{i}(E(u_{1})), \sup_{i}h_{i}(E(u_{2}))\right)\nonumber\\&\leq&  h (E(u_{2}))+t\varphi\left(h(E(u_{1})), (E(u_{2}))\right).\nonumber 
	\end{eqnarray}
	This implies that $ h $ is a geodesic $ \varphi_{E} $-convex function.
\end{proof}
\begin{theorem}
	The function $ h:C\longrightarrow \mathbb{R} $ is a geodesic $ \varphi_{E} $-convex, where $ C $ is a geodesic $ E $-convex set. The inequality $ \varphi(h(E(\mu),h(E(\mu^{*}))\geqslant 0, \forall E(\mu)\in C $ is necessary for $ h $ to have a local minimum at $ E(\mu^{*})\in Int(C) $.
\end{theorem}
\begin{proof}
	Due to the fact that $ C $ is a geodesic E-convex set and $ E(\mu^{*})\in Int(C) $, then $ B(E(\mu^{*}),S)\subset C $ for some $ s>0 $. Let $ E(\mu)\in C $, then 
	$$	h\left(\gamma_{E(\mu),E(\mu^{*})} (t)\right)\leq h(E(\mu^{*}))+t\varphi\left(h(E(\mu)), (E(\mu^{*}))\right).$$
	Since $ h $ attains its local minimum at $ E(\mu^{*}) $, then
	\begin{equation}\label{eq11}
	h(E(\mu^{*}))\leq h\left(\gamma_{E(\mu),E(\mu^{*})} (\zeta)\right),
	\end{equation} 
	where $ \zeta\in\left(\left.0,1 \right]  \right. $ such that $ h\left(\gamma_{E(\mu),E(\mu^{*})} (t)\right) \in B(E(\mu^{*}),S),\ \ \forall t\in [0,\zeta] $.
	Also,
	\begin{eqnarray}\label{eq12}
	h\left(\gamma_{E(\mu),E(\mu^{*})} (\zeta)\right) \leq h(E(\mu^{*}))+\zeta \varphi (h(E(\mu)),h(E(\mu^{*})),
	\end{eqnarray}
	then from (\ref{eq11}) and (\ref{eq12}), we obtain
	$ \varphi(h(E(\mu)),h(E(\mu^{*}))\geqslant 0,\ \ \forall E(\mu)\in C $.
\end{proof}
\begin{theorem}
	The function $ h:B\longrightarrow \mathbb{R} $ is a geodesic $ \varphi_{E} $-convex, where $ B $ is a geodesic $ E $-convex set and $ \varphi $ is bounded from above on $ h(B)\times h(B) $ with an upper bound $ K $ w.r.t.E. Then $ h $ is continuous on $ Int(B) $.
\end{theorem}
\begin{proof} 
	Assume that $ E(u)\in Int(B) $ and $ (U,\psi) $ is a chart containing $ E(u) $. Since $ \psi $ is a diffemorphism and by using Theorem \ref{th1} and Theorem \ref{th2}, we get 
	$ ho\psi^{-1}: \psi(U\cap Int(B))\longrightarrow \mathbb{R} $ as also geodesic $ \varphi_{E} $-convex and then it is continuous.
	Hence, 
	$ h=ho\psi^{-1}o\psi : (U\cap Int(B))\longrightarrow \mathbb{R} $ is continuous.\\
	Also, since $ E(u) $ is arbitrary, then $ h $ is continuous on $ Int(B) $.
\end{proof}
From the definition of geodesic $ \varphi_{E} $ -convex, we obtain the following proposition.
\begin{proposition}
	Assume that $ \left\lbrace \varphi^{i}:i\in \mathbb{N}\right\rbrace $  is a collection of bifunctions such that $ h: B\longrightarrow \mathbb{R} $ is a  geodesic $ \varphi_{E}^{i} $-convex function  for each $ i $. If $ \varphi^{i}\longrightarrow\varphi $ as $ i\longrightarrow\infty $, then $ h $ is also a geodesic $ \varphi_{E} $-convex function.
\end{proposition}
As a special case in the above proposition, we have the following proposition.
\begin{proposition}
	Assume that $ \left\lbrace \varphi^{i}:i\in \mathbb{N}\right\rbrace $  is a collection of bifunctions such that $ h: B\longrightarrow \mathbb{R} $ is a  geodesic $ \varphi_{E}^{*} $-convex function  where $ \varphi_{E}^{*}=\sum_{l=1}^{i}\varphi_{E}^{l} $. If $ \varphi_{E}^{*} $ converges to $ \varphi_{E} $, then $ h $ is also a geodesic $ \varphi_{E} $-convex function.
\end{proposition}
\begin{theorem}
	Consider $ h:B\longrightarrow \mathbb{R} $ to be a strictly geodesic $ \varphi_{E} $-convex, where $ B $ is a  geodesic $ E $-convex set and $ \varphi $ is antisymmetric function w.r.t. $ E $. Then $$dh_{E(\mu_{1})}\dot{\gamma}_{E(\mu_{1}),E(\mu_{2})}\neq dh_{E(\mu_{2})}\dot{\gamma}_{E(\mu_{1}),E(\mu_{2})},$$$  \forall E(\mu_{1}),E(\mu_{2})\in B $  and $ E(\mu_{1})\neq E(\mu_{2}) $.
\end{theorem}
\begin{proof}
	Since $ \gamma_{E(\mu_{2}),E(\mu_{1})}(t)= \gamma_{E(\mu_{1}),E(\mu_{2})}(1-t), \ \ \forall t\in[0,1] $, then
$$	dh_{E(\mu_{2})}\dot{\gamma}_{E(\mu_{2}),E(\mu_{1})} =-dh_{E(\mu_{2})}\dot{\gamma}_{E(\mu_{1}),E(\mu_{2})}. $$
	By contradition, let $$ dh_{E(\mu_{1})}\dot{\gamma}_{E(\mu_{1}),E(\mu_{2})} =dh_{E(\mu_{2})}\dot{\gamma}_{E(\mu_{1}),E(\mu_{2})}, $$
	but  $ h $ is a geodesic $ \varphi_{E} $-convex function, then 
	\begin{equation}\label{eq13}
	dh_{E(\mu_{1})}\dot{\gamma}_{E(\mu_{1}),E(\mu_{2})}<\varphi (h(E(\mu_{1})),h(E(\mu_{2})) .
	\end{equation}
Also,	
	$$dh_{E(\mu_{2})}\dot{\gamma}_{E(\mu_{2}),E(\mu_{1})}<\varphi (h(E(\mu_{2})),h(E(\mu_{1})).$$\\
	
	On the other hand, 
	
$$	dh_{E(\mu_{2})}\dot{\gamma}_{E(\mu_{2}),E(\mu_{1})} =-dh_{E(\mu_{2})}\dot{\gamma}_{E(\mu_{1}),E(\mu_{2})},$$
	then 
	\begin{eqnarray}\label{eq14}
	 -dh_{E(\mu_{2})}\dot{\gamma}_{E(\mu_{1}),E(\mu_{2})}<\varphi (h(E(\mu_{2})),h(E(\mu_{1})).
	 \end{eqnarray}
	 Moreover, since $ \varphi $ is antisymmetry function, then (\ref{eq14}) becomes  
	 $$dh_{E(\mu_{2})}\dot{\gamma}_{E(\mu_{1}),E(\mu_{2})}>\varphi (h(E(\mu_{1})),h(E(\mu_{2})),$$
	 hence,
	 \begin{equation}\label{eq15}
	  dh_{E(\mu_{1})}\dot{\gamma}_{E(\mu_{1}),E(\mu_{2})}>\varphi (h(E(\mu_{1})),h(E(\mu_{2})).
	  	 \end{equation}
	  From (\ref{eq13}) and (\ref{eq15}), we obtain a contradition, then 
	  $dh_{E(\mu_{1})}\dot{\gamma}_{E(\mu_{1}),E(\mu_{2})}\neq dh_{E(\mu_{2})}\dot{\gamma}_{E(\mu_{1}),E(\mu_{2})}.$  
\end{proof}
\section{$ \varphi_{E} $-Epigraphs}
In this section, $ \varphi_{E} $- epigraphs are introduced on complete Riemannian manifolds, and a characterization of geodesic $ \varphi_{E} $-convex functions in terms of their $ \varphi_{E} $-epigraphs is obtained.
\begin{definition}
	A set $ B\subset N\times \mathbb{R} $ is called a geodesic $ \varphi_{E} $-convex set if 
	$$\left(\gamma_{E(u_{1}),E(u_{2})}(t),v_{2}+t\varphi(v_{1},v_{2}) \right)\in B, \ \ \forall (u_{i},v_{i})\in B , t\in[0,1]. $$
\end{definition}
Therefore, a $ \varphi_{E} $- epigraph of function $ h $ is defined by 
$$ epi_{\varphi_{E}}(h)=\left\lbrace (u,v)\in E(N)\times \mathbb{R}: h(u)\leq v\right\rbrace .$$ 
\begin{theorem}
	Consider $ B\subset N $ to be a geodesic $ E $-convex set, and $ \varphi $ is non-decreasing. The set $ epi_{\varphi_{E}}(h)$ is a geodesic $ \varphi_{E} $-convex, if and only if $ h:B\longrightarrow \mathbb{R }$ is a geodesic $ \varphi_{E} $-convex function.
\end{theorem}
\begin{proof}
	Let $ u_{1},u_{2}\in B $ and $ t\in [0,1] $, and since $ B $ is an $ E $-convex set, then $ E(u_{1}),E(u_{2})\in E(B)\subseteq B $. Hence, $$ (E(u_{1}),h(E(u_{1}))),(E(u_{2}),h(E(u_{2})))\in epi_{\varphi_{E}}(h) .$$ Due to the fact that $ epi_{\varphi_{E}} (h) $ is a geodesic $ \varphi_{E} $-convex set, then $$\left(\gamma_{E(u_{1}),E(u_{2})}(t), h(E(u_{2}))+t\varphi h(E(u_{1})),h(E(u_{2})) \right) \in epi_{\varphi_{E}}(h).$$
	This implies to 
$ h\left( 	\gamma_{E(u_{1}),E(u_{2})}(t)\right) \leq h(E(u_{2}))+t\varphi h(E(u_{1})),h(E(u_{2})) $.
Consequently, $ h $ is a geodesic $ \varphi_{E} $-convex function. \\

Now, let us condiser that $ (u_{1}^{*},v_{1}), (u_{2}^{*},v_{2})\in epi_{\varphi_{E}}(h) $, then $ u_{1}^{*}, u_{1}^{*}\in E(B) $ which means that there are $ u_{1},u_{2}\in B $ such that $ E(u_{1})=u_{1}^{*} $ and $ E(u_{2})=u_{2}^{*} $. Hence, $ h(E(u_{1}))\leq v_{1}, h(E(u_{2}))\leq v_{2} $ and, since $ h $ is a geodesic $ \varphi_{E} $-convex function, then 
 \begin{eqnarray}
 h\left(\gamma_{E(u_{1}),E(u_{2})}(t)  \right) &\leq&  h(E(u_{2}))+t\varphi h(E(u_{1})),h(E(u_{2}))\nonumber\\ &\leq& v_{2}+t\varphi(v_{1},v_{2}),\nonumber
 \end{eqnarray}
 which implies that
 $ \left(\gamma_{E(u_{1}),E(u_{2})}(t), v_{2}+t\varphi(v_{1},v_{2})\right)\in epi_{\varphi_{E}}(h),\ \ \forall t\in[0,1] $. That is, $ epi_{\varphi_{E}}(h) $ is a geodesic $ \varphi_{E} $-convex set. 
\end{proof}
\begin{theorem}
	Consider $ \left\lbrace B_{i}, i\in I\right\rbrace $ to be a family of geodesic $ \varphi_{E} $-convex sets, then $ B=\cap_{i\in I}B_{i} $  is also a geodesic $ \varphi_{E} $-convex set.
\end{theorem}
\begin{proof}
Let $ (\mu_{1},\nu_{1}), (\mu_{2},\nu_{2})\in \cap_{i\in I}B_{i} $, then $ (\mu_{1},\nu_{1}), (\mu_{2},\nu_{2})\in B_{i}, \forall i\in I $. Hence, 
$$\left(\gamma_{E(\mu_{1}),E(\mu_{2})}(t),\nu_{2}+t\varphi (\nu_{1},\nu_{2}) \right) \in B_{i} \ \  \forall  t\in[0,1].$$
$ \Rightarrow $
$$\left(\gamma_{E(\mu_{1}),E(\mu_{2})}(t),\nu_{2}+t\varphi (\nu_{1},\nu_{2}) \right) \in \cap_{i\in I}B_{i} \ \ \forall t\in[0,1].$$
This implies $ \cap_{i\in I}B_{i} $ is a geodesic $ \varphi_{E} $-convex function.
\end{proof}
By using the above theorem, we can obtain the following corollary
\begin{corollary}
Let $ \left\lbrace h_{i}, i\in I\right\rbrace $ be a family of geodesic $ \varphi_{E} $-convex functions defined on a geodesic $ E $-convex set $ B\subset N $ which is bounded above, and $ \varphi $ is non-decreasing. If the $ E $-epigraphs $ epi_{\varphi_{E}}(h_{i}) $ are geodesic $ \varphi_{E} $-convex sets, then $ h=\sup_{i\in I} h_{i} $ is also a geodesic $ \varphi_{E} $-convex function on $ B $.
\end{corollary}

\section*{Data availability statement}

My manuscribt has no associated data


\end{document}